\documentclass[preprint,12pt]{elsarticle}

\usepackage{amssymb}
\usepackage{amsmath}
\usepackage{float}
\DeclareMathOperator{\supp}{supp}
\usepackage{tabularx}
\usepackage{parskip}     
\usepackage{fancyhdr}    
\usepackage{enumitem}
\usepackage{xcolor}
\usepackage{emptypage}   
\usepackage[small,hang,bf]{caption}    
\usepackage{subcaption}
\usepackage{amsthm}
\newtheorem{theorem}{Theorem}[section]
\newtheorem{lemma}[theorem]{Lemma}
\newtheorem{proposition}[theorem]{Proposition}
\newtheorem{definition}[theorem]{Definition}

\newtheorem{corollary}{Corollary}[theorem]
\newtheorem{example}[theorem]{Example}
\newtheorem{remark}[theorem]{Remark}

\journal{ASCM}

\begin{document}

\begin{frontmatter}


\author[1]{Francisco Alejandro Villegas Acuña\corref{cor1}}
\cortext[cor1]{Corresponding author. Email: francisco.villegas@unison.mx}

\affiliation[1]{organization={Departamento de Matemáticas, Universidad de Sonora},
            city={Hermosillo},
            postcode={83000},
            state={Sonora},
            country={México}}

\title{Bourgain-Morrey sequence spaces: structural properties, relations to classical \(\ell^{p}\) spaces and duality}



\begin{abstract}
We study the discrete Bourgain-Morrey sequence spaces $\ell^{p}_{q,r}(\mathbb{Z})$, 
recently introduced as discrete counterparts of Morrey-type spaces. 
We show that $c_{00}$ is dense in $\ell^{p}_{q,r}$, hence the spaces are separable. 
We establish embeddings $\ell^{1}\hookrightarrow \ell^{p}_{q,r}\hookrightarrow \ell^{r}$ for $r>1$, 
while for $r=1$ one has $\ell^{p}_{q,1}=\ell^{1}$. 
For each $p$, the identity $\ell^{p}_{q,p}=\ell^{p}$ yields uncountably many equivalent norms on $\ell^{p}$. 
We also introduce a block space as a natural predual of $\ell^{p}_{q,r}$ and prove the duality 
$(\ell^{p}_{q,r})^{*}=\mathrm{h}^{p'}_{q',r'}$, from which reflexivity follows for 
$1<p<q<\infty$ and $1<r<\infty$. 
This work completes the foundational stage of the discrete Bourgain–Morrey theory by fully characterizing its structure and duality.
\end{abstract}
\begin{keyword}
Bourgain-Morrey spaces, sequence spaces, equivalent norms, duality, reflexivity.\\
MSC 2020: 46B45, 46B10, 46B15, 46B70, 46E30, 42B35.


\end{keyword}

\end{frontmatter}



\section{Introduction}
\label{sec1}
The study of Banach sequence spaces has played a central role in functional analysis since the classical introduction of $\ell^p$ spaces (see \cite{lindenstrauss1977classical}). Beyond their intrinsic geometric interest, sequence spaces provide tractable models for testing operator theory, interpolation (see \cite{bennett1988interpolation}), and duality phenomena, and often serve as discrete analogues of function spaces arising in harmonic analysis. In this framework, generalized constructions such as Lorentz (see \cite{CiesielskiLewicki2019}), Orlicz (see \cite{rao1991theory} and \cite{TripathyAltinEtAl2008}), and Morrey-type sequence spaces have been widely investigated (see \cite{morrey1938}, \cite{gunawan2004discrete}, \cite{HaroskeSkrzypczak2025} and \cite{MorreySeq2020}). 

Motivated by the rich theory of Bourgain–Morrey function spaces and their recent extensions—including Banach space–valued frameworks~\cite{preduals of Banach space valued}, Besov and Campanato-type generalizations~\cite{Besov–Bourgain–Morrey–Campanato,Bourgain–Morrey Spaces Mixed with}, and grand or variable versions~\cite{Besov–Bourgain–Morrey spaces}—the discrete Bourgain-Morrey sequence spaces $\ell^{p}_{q,r}$ were recently introduced in \cite{GuzmanSanMartinVillegas}. There, the authors established foundational properties, including equivalent norms and basic convolution inequalities.
This paper aims to complete the basic functional-analytic picture of these spaces. We present a systematic study that resolves fundamental questions about their structure, their relation to classical spaces, and their duality.

Specifically, we prove: continuous and proper embeddings between classical $\ell^1$, $\ell^{p}_{q,r}$, and $\ell^r$; complete identification with $\ell^p$ in the case $r=p$; separability via density of $c_{00}$; and that the canonical unit vectors form a $1$--unconditional Schauder basis of $\ell^{p}_{q,r}$ (in particular, the coordinate projections are uniformly bounded).

Moreover, since discrete Bourgain-Morrey spaces admit three different equivalent norms (see \cite{GuzmanSanMartinVillegas}), the identification $\ell^{p,q}_p \equiv \ell^p$ immediately yields three uncountable families of equivalent norms on the classical spaces $\ell^p$.  We do not address here whether these norms are genuinely new or distinct from those already considered in the literature.

In addition, we develop the duality theory of $\ell^{p}_{q,r}$ by introducing a natural block space $\mathrm{h}^{p'}_{q',r'}$ and proving that $(\ell^{p}_{q,r})^* \simeq \mathrm{h}^{p'}_{q',r'}$ whenever $1<p<q<\infty$ and $1<r<\infty$. This identification also shows that $\ell^{p}_{q,r}$ is reflexive in this range, thereby extending the duality framework of Morrey-type spaces to the discrete Bourgain-Morrey setting. 

To demonstrate the practical utility of this framework, we include a final section on applications to operators and difference equations.

These results situate $\ell^{p}_{q,r}$ as a well-structured and flexible family of sequence spaces, aligned with classical theory yet rich enough to generate new perspectives on renormings, embeddings, and duality. We believe they provide a natural foundation for further work in operator theory and interpolation in discrete settings.


\section{Main results}

This section gathers the principal contributions of the paper in a concise form. 
For clarity of exposition, we present the most relevant structural and duality 
properties of the Bourgain-Morrey sequence spaces in a self-contained manner. 
The results stated here serve as a reference point for the subsequent sections, 
where detailed proofs and further discussions are provided.

The following theorem shows that, in analogy with the classical $\ell^p$ spaces, 
the finitely supported sequences are dense in the Bourgain-Morrey setting. 
This ensures that $\ell^{p}_{q,r}$ inherits separability, a key structural feature 
for further functional-analytic developments.
\begin{theorem}[Density of $c_{00}$ in $\ell^{p}_{q,r}$]\label{Density}
Let $1\leq p<q<\infty$ and $1\leq r<\infty$, then
\begin{equation*}
\overline{c_{00}}^{\,\|\cdot\|_{\ell^{p}_{q,r}}}=\ell^{p}_{q,r}.
\end{equation*}
In particular, $\ell^p_{q,r}$ is separable.
\end{theorem}
The next proposition identifies the position of $\ell^{p}_{q,r}$ with respect 
to classical sequence spaces. In particular, it lies strictly between $\ell^1$ 
and $\ell^r$, with precise information on the endpoint $r=1$.
\begin{proposition}
    Let $1\leq p<q<\infty$, then $\ell^{1}\subsetneq \ell^{p}_{q,r}\subsetneq \ell^{r}$ with continuous inclusion for $r>1$, proper inclusion for $r>q$, and $\ell^{p}_{q,r}=\ell^{1}$ for $r=1$.
\end{proposition}

A remarkable feature occurs in the case $r=p$: the Bourgain-Morrey space 
collapses to the classical $\ell^p$, providing a wealth of new perspectives on 
this well-known space.
\begin{theorem}\label{equivalent norms}
    Let $1\leq p<\infty$. Then for all $q>p$ we have that $\ell^{p}_{q,p}$ and $\ell^{p}$ are equal in the sense of norms.
\end{theorem}
As a consequence, one obtains an uncountable collection of equivalent norms 
on $\ell^p$. This illustrates how the Bourgain-Morrey construction generates 
renormings of classical sequence spaces, some of which appear structurally 
different from the usual one.
\begin{corollary}
    For $1\leq p<\infty$, the space $\ell^{p}$ possesses an uncountable families of equivalent norms. Namely
    \begin{equation*}
\left[\sum\limits_{m\in\mathbb{Z},\,N\in\mathbb{N}_{0}}
(2N+1) ^{\frac{p}{q}-1}
\sum\limits_{l\in S_{m,N}  }
\left\vert x(l)\right\vert ^{p}\right]  ^{1/p},\text{ for every } q>p.
    \end{equation*}
\end{corollary}
The central novelty of the present work lies in the duality theory. 
We show that $\ell^{p}_{q,r}$ admits a natural block predual, which 
leads to a complete identification of its dual space.
\begin{theorem}[Duality]\label{thm:duality-main}
Let $1<p<q<\infty$ and $1<r<\infty$. Then the Bourgain-Morrey sequence space 
$\ell^{p}_{q,r}(\mathbb{Z})$ admits a natural block predual $h^{p'}_{q',r'}(\mathbb{Z})$, and one has the duality identifications
\[
(\ell^{p}_{q,r})^* \cong h^{p'}_{q',r'} 
\qquad \text{and} \qquad
(h^{p'}_{q',r'})^* \cong \ell^{p}_{q,r},
\]
with isometric isomorphisms.
\end{theorem}
Finally, the previous theorem immediately yields the reflexivity of 
$\ell^{p}_{q,r}$ in the interior range of parameters, mirroring the classical 
$\ell^p$ theory.

\begin{remark}[Final summary of the duality picture]
Table~1 encapsulates the outcome of our main results across the full range of parameters. 
It highlights the different regimes where the Bourgain-Morrey sequence spaces either 
collapse to classical sequence spaces (such as $\ell^1$, $\ell^p$, or $\ell^r$), 
give rise to genuinely new Banach spaces with a nontrivial block predual, 
or remain partially unexplored (as in the case $r=\infty$, where the duality 
picture is still open). 

This panorama illustrates that the discrete Bourgain-Morrey construction is robust: 
it reproduces well-known spaces in the expected limits, provides novel renormings 
of classical $\ell^p$ spaces, and yields a full-fledged duality theory in the 
interior range $1<p<q<\infty$, $1<r<\infty$. In particular, the block space 
$h^{p'}_{q',r'}$ emerges naturally as the predual, thereby extending to the 
discrete setting the philosophy developed for continuous Bourgain-Morrey spaces. 

Taken together, these results show that the class $\ell^p_{q,r}(\mathbb{Z})$ 
offers a coherent and versatile framework, unifying classical sequence spaces, 
Morrey-type constructions, and new duality phenomena within a single scheme.
\end{remark}
To conclude, the above results provide a unified picture of the Bourgain-Morrey 
sequence spaces, clarifying their relation to classical sequence spaces and 
establishing their fundamental duality properties. This completes the summary of 
our main findings and sets the stage for the detailed analysis carried out in the 
subsequent sections.

\section{Preliminaries}
\subsection{Notation and basic definitions}
\label{subsec:notation}
We use standard notation, and denote by $C$ a constant that may change from line to line. Sometimes we write $C_{a,b}$ to indicate dependence on parameters $a$ and $b$.
We begin by recalling two fundamental spaces in the theory of sequence spaces:

\begin{definition}[Space of finitely supported sequences]
The space $c_{00}(\mathbb{Z})$ consists of all real-valued sequences with finite support:
\[
c_{00}(\mathbb{Z}) := \left\{ (x_k)_{k \in \mathbb{Z}} \in \mathbb{R}^{\mathbb{Z}} : \exists\, N \in \mathbb{N} \text{ such that } x_k = 0 \text{ for all } |k| > N \right\}.
\]
\end{definition}

\begin{definition}[$\ell^p$ spaces]
For $1 \leq p < \infty$, the space $\ell^p(\mathbb{Z})$ consists of all real-valued sequences with finite $p$-norm:
\[
\ell^p(\mathbb{Z}) := \left\{ (x_k)_{k \in \mathbb{Z}} \in \mathbb{R}^{\mathbb{Z}} : \|x\|_{\ell^{p}} := \left( \sum_{k \in \mathbb{Z}} |x(k)|^p \right)^{1/p} < \infty \right\}.
\]
For $p = \infty$, we define:
\[
\ell^\infty(\mathbb{Z}) := \left\{ (x_k)_{k \in \mathbb{Z}} \in \mathbb{R}^{\mathbb{Z}} : \|x\|_{\ell^{\infty}} := \sup_{k \in \mathbb{Z}} |x(k)| < \infty \right\}.
\]
\end{definition}

\begin{definition}[Canonical unit sequences]\label{canonical sequences}
For each $j \in \mathbb{Z}$, the \emph{canonical unit sequence} $e^j = (e^j_k)_{k \in \mathbb{Z}}$ is defined by:
\[
e^j_k := \delta_{jk} = 
\begin{cases}
1 & \text{if } k = j \\
0 & \text{if } k \neq j
\end{cases}
\quad \text{for } k \in \mathbb{Z}.
\]
The family $\{e^j\}_{j \in \mathbb{Z}}$ forms a Schauder basis for many classical sequence spaces, including $c_{00}$ and $\ell^p$ for $1 \leq p < \infty$.
\end{definition}

These elementary concepts play a crucial role in our analysis of discrete Bourgain-Morrey spaces, particularly in establishing density results and structural properties.
\label{subsec1}
\subsection{Previous Results on Discrete Bourgain-Morrey Spaces}

In our previous article \cite{GuzmanSanMartinVillegas}, we introduced the discrete Bourgain–Morrey spaces $\ell^{p}_{q,r}(\mathbb{Z})$
and established their foundational theory, including norm equivalences, basic properties, and convolution inequalities. To make the present work self-contained, we recall below the key definitions and results from \cite{GuzmanSanMartinVillegas}. All material in this section is taken directly from \cite{GuzmanSanMartinVillegas}, and proofs are omitted here but can be found in the cited reference. New contributions in this paper begin in Section 4, where we prove original results on separability, embeddings, and identifications with classical spaces.

We now recall the definition of the Bourgain–Morrey sequence spaces:

        \begin{definition}
Let $1\leq p < q\leq\infty$ and let $1\leq r<\infty$. We denote by $\ell^{p}_{q,r}(\mathbb{Z})=\ell^{p}_{q,r}$ the set of sequences $x=(x_k)_{k\in \mathbb{Z}}$ with values in $\mathbb{R}$ such that
\[
\|x\|_{\ell^{p}_{q,r}}:=\left(\sum_{m\in \mathbb{Z},\,N\in \mathbb{N}_0}|S_{m,N}|^{\frac{r}{q}-\frac{r}{p}}\left(\sum_{k\in S_{m,N}}|x(k)|^p\right)^{\frac{r}{p}}\right)^{\frac{1}{r}}<\infty,
\]
where $\mathbb{N}_0 = \mathbb{N} \cup \{0\}$ and $S_{m,N} = \{m-N, m-N+1, \dots, m+N\}$ with $|S_{m,N}| = 2N+1$.
\end{definition}
We begin by defining the three equivalent norms on the space $\ell^{p}_{q,r}(\mathbb{Z})$.

\begin{definition}[Equivalent Norms]\label{def:equivalent_norms}
Let $1\leq p< q<\infty$, $1\leq r<\infty$, and let $x \in \mathbb{R}^{\mathbb{Z}}$.
\begin{enumerate}
    \item[(a)] \textbf{(Centered Intervals)}
    We denote by $\mathcal{S} = \{ S_{m,{N}} \}_{m \in \mathbb{Z}, N \in \mathbb{N}_{0}}$.     The associated norm to $\mathcal{S}$ is defined by
    \begin{equation}\label{eq:norm_S}
        \|x\|^{\mathcal{S}}_{\ell^{p}_{q,r}} := \left( \sum_{N = 0}^{\infty} \sum_{m \in \mathbb{Z}} (2N+1)^{r\left(\frac{1}{q} - \frac{1}{p}\right)} \left( \sum_{l \in S_{m,N}} |x(l)|^{p} \right)^{\frac{r}{p}} \right)^{\frac{1}{r}}.
    \end{equation}

    \item[(b)] \textbf{(Dyadic Intervals)} For $j \in \mathbb{N}_{0}$ and $k \in \mathbb{Z}$, define the discrete dyadic interval
    \[
        I(j,k) := [2^{j}k, 2^{j}(k+1)) \cap \mathbb{Z}.
    \]
     We denote by $\mathcal{D} = \{ I(j,k) \}_{ j \in \mathbb{N}_{0}, k \in \mathbb{Z}}$.
    The associated norm to $\mathcal{D}$ is defined by
    \begin{equation}\label{eq:norm_D}
        \|x\|^{\mathcal{D}}_{\ell^{p}_{q,r}} := \left( \sum_{j = 0}^{\infty} \sum_{k \in \mathbb{Z}} (2^{j})^{r\left(\frac{1}{q} - \frac{1}{p}\right)} \left( \sum_{l \in I(j,k)} |x(l)|^{p} \right)^{\frac{r}{p}} \right)^{\frac{1}{r}}.
    \end{equation}

    \item[(c)] \textbf{(Dyadic Length Intervals)} As a subfamily of the centered intervals, consider those with dyadic side lengths:
    \[
        \mathcal{A} = \{ S_{m,2^{N}} \}_{m \in \mathbb{Z}, N \in \mathbb{N}_{0}} \cup \{ \{m\} \}_{m \in \mathbb{Z}}.
    \]
    The associated norm to $\mathcal{A}$ is defined by
    \begin{equation}\label{eq:norm_A}
        \|x\|^{\mathcal{A}}_{\ell^{p}_{q,r}} := \left( \sum_{N = 0}^{\infty} \sum_{m \in \mathbb{Z}} (2^{N+1}+1)^{r\left(\frac{1}{q} - \frac{1}{p}\right)} \left( \sum_{l \in S_{m,2^{N}}} |x(l)|^{p} \right)^{\frac{r}{p}} \right)^{\frac{1}{r}}.
    \end{equation}
\end{enumerate}
\end{definition}

A fundamental result is that these different definitions yield the same space with equivalent norms.
\begin{lemma}[Norm equivalence {\cite[Theorem~3.1]{GuzmanSanMartinVillegas}}]\label{lem:norm_equivalence}
Let $1\le p<q<\infty$ and $1\le r<\infty$. The norms
$\|\cdot\|^{\mathcal A}_{\ell^{p}_{q,r}}$, $\|\cdot\|^{\mathcal D}_{\ell^{p}_{q,r}}$, and
$\|\cdot\|^{\mathcal S}_{\ell^{p}_{q,r}}$ are equivalent.
\end{lemma}

Consequently, the space
\[
\ell^{p}_{q,r}(\mathbb{Z})
:=\Big\{x\in\mathbb{R}^{\mathbb{Z}}:\ \|x\|_{\ell^{p}_{q,r}}<\infty\Big\}
\]
is well-defined, where $\|\cdot\|_{\ell^{p}_{q,r}}$ denotes any of the equivalent norms above.\\
\noindent
We will rely on two structural properties of $\ell^{p}_{q,r}(\mathbb Z)$ that are repeatedly used throughout the paper:
translation invariance and stability under convolution with $\ell^1$ kernels.

\begin{lemma}[Translation and convolution invariance {\cite[Proposition~2.2]{GuzmanSanMartinVillegas}}]\label{lem:trans_conv}
Let $1 \le p < q < \infty$ and $1 \le r < \infty$.
\begin{enumerate}
\item[(i)] \textbf{(Translation invariance)} For any $t \in \mathbb Z$ and $x \in \ell^{p}_{q,r}(\mathbb Z)$,
\[
\|x(\cdot-t)\|_{\ell^{p}_{q,r}}=\|x\|_{\ell^{p}_{q,r}}.
\]
\item[(ii)] \textbf{(Convolution inequality)} If $y\in \ell^{1}(\mathbb Z)$ and $x\in \ell^{p}_{q,r}(\mathbb Z)$, then
$x*y\in \ell^{p}_{q,r}(\mathbb Z)$ and
\[
\|x*y\|_{\ell^{p}_{q,r}} \le \|y\|_{\ell^{1}}\,\|x\|_{\ell^{p}_{q,r}}.
\]
\end{enumerate}
\end{lemma}

Finally, we recall a key example that establishes the non-triviality of these spaces. 

\begin{example}[Non-Triviality]\label{ex:non_trivial}
Let $1 \leq p < q < \infty$ and $1 \leq r < \infty$. Let $e^{0} = (e^{0}_{k})_{k \in \mathbb{Z}}$ be the canonical unit sequence defined in \ref{canonical sequences}. Then $e^{0} \in \ell^{p}_{q,r}(\mathbb{Z})$ and its norm, computed with respect to the dyadic norm $\| \cdot \|^{\mathcal{D}}_{\ell^{p}_{q,r}}$, is given by
\[
    \| e^{0} \|^{\mathcal{D}}_{\ell^{p}_{q,r}} = \left( \sum_{j=0}^{\infty} \left( 2^{r \left(\frac{1}{q} - \frac{1}{p}\right)} \right)^{j} \right)^{1/r}.
\]
This series converges since $\frac{1}{q} - \frac{1}{p} < 0$.
\end{example}
This concludes the summary of the foundational results from~\cite{GuzmanSanMartinVillegas}. The new contributions of the present article begin in Section~\ref{sec:properties}.

\section{Structural properties of $\ell^{p}_{q,r}$}
\label{sec:properties}

In this section, we investigate fundamental structural properties of the discrete Bourgain-Morrey spaces $\ell^{p}_{q,r}(\mathbb{Z})$. These spaces, introduced as discrete analogues of Bourgain-Morrey function spaces, exhibit rich geometric and functional-analytic characteristics. We establish several key results:
\begin{itemize}
    \item The density of finitely supported sequences ($c_{00}$) in $\ell^{p}_{q,r}$, implying separability;
    \item Continuous embeddings relating $\ell^{p}_{q,r}$ to classical sequence spaces $\ell^1$ and $\ell^r$;
    \item Complete identification of $\ell^{p}_{q,1}$ with $\ell^1$ (with equivalent norms);
    \item Non-trivial examples demonstrating proper inclusions and limiting cases.
\end{itemize}
These properties not only clarify the internal structure of $\ell^{p}_{q,r}$ but also reveal its connections to well-established Banach spaces. In particular, the embeddings $\ell^1 \hookrightarrow \ell^{p}_{q,r} \hookrightarrow \ell^r$ provide a natural framework for interpolation and functional analysis in these spaces.\\

\begin{proof}[Proof of Theorem \ref{Density}]
   Since $e^{0}\in \ell^{p}_{q,r}$ (by Example \ref{ex:non_trivial}), it follows from item 1 in Lemma \ref{lem:trans_conv}, that 
    \begin{equation*}
        \|e^{k}\|_{\ell^{p}_{q,r}}= \|e^{0}(\cdot-(-k)\|_{\ell^{p}_{q,r}}=\|e^{0}\|_{\ell^{p}_{q,r}}<\infty,
    \end{equation*}
for every $k\in \mathbb{Z}$.

Now, let $x\in c_{00}$ be arbitrary. Then there exist $M\in \mathbb{N}$ and $\alpha_{i}\in \mathbb{R}$, for $i\in \{1,...,M\}$, such that $x=\sum_{i=1}^{M}\alpha_{i}e^{i}$. By triangle inequality it follows that
\begin{equation}
    \|x\|_{\ell^{p}_{q,r}}=\left\|\sum_{i=1}^{M}\alpha_{i}e^{i}\right\|_{\ell^{p}_{q,r}}\leq \sum_{i=1}^{M}|\alpha_{i}|\|e^{i}\|_{\ell^{p}_{q,r}}<\infty.
\end{equation}
Hence $c_{00}\subset \ell^{p}_{q,r}$. Now fix $x\in \ell^{p}_{q,r}(\mathbb{Z})$ and $\varepsilon>0$. For each $(j,k)\in \mathbb{N}_{0}\times \mathbb{Z}$, set
\[
a_{j,k}(x):=|I(j,k)|^{\frac{r}{q}-\frac{r}{p}}\Big(\sum_{l\in I(j,k)} |x_l|^p\Big)^{\frac{r}{p}}\ge 0.
\]
Since $\|x\|_{\ell^{p}_{q,r}}^r=\sum_{(j,k)\in \mathbb{N}_{0}\times\mathbb{Z}} a_{j,k}(x)<\infty$, there exists a finite set $F\subset \mathbb{N}_0\times\mathbb{Z}$ such that
\[
\sum_{(j,k)\notin F} a_{j,k}(x)<\varepsilon^r.
\]
Let $U:=\bigcup_{(j,k)\in F} I(j,k)$, which is a finite subset of $\mathbb{Z}$ (finite union of finite intervals). Define $y:=x\cdot \mathbf{1}_U$,  (where $\mathbf{1}$ denotes indicator function) so $y\in c_{00}$.
Then, for $(j,k)\in F$ we have $I(j,k)\subset U$ and hence
\[
\sum_{l\in I(j,k)} |(x-y)(l)|^p=0.
\]
For $(j,k)\notin F$, by monotonicity,
\[
\sum_{l\in I(j,k)} |(x-y)(l)|^p
=\sum_{l\in I(j,k)\setminus U} |x(l)|^p
\le \sum_{l\in I(j,k)} |x(l)|^p.
\]
Therefore,
\[
\|x-y\|_{\ell^{p}_{q,r}}^r
=\sum_{(j,k)\in \mathbb{N}_{0}\times\mathbb{Z}}|I(j,k)|^{\frac{r}{q}-\frac{r}{p}}
\Big(\sum_{l\in I(j,k)} |(x-y)(l)|^p\Big)^{\frac{r}{p}}
\le \sum_{(j,k)\notin F} a_{j,k}(x)
<\varepsilon^r.
\]
Hence $\|x-y\|_{\ell^{p}_{q,r}}<\varepsilon$ with $y\in c_{00}$, which proves the density.
\end{proof}
\begin{corollary}[Separability]
For $1 \le p < q < \infty$ and $1 \le r < \infty$, the space $\ell^p_{q,r}(\mathbb Z)$ is separable.
\end{corollary}

\begin{proof}
By Theorem \ref{Density} we have that $c_{00}$ is dense in $\ell^p_{q,r}$. Since there exists a countable dense subset in $c_{00}$, the space $\ell^p_{q,r}$ is separable.
\end{proof}

Next we relate $\ell^{p}_{q,r}$ to the classical sequence spaces. 
The following inclusion shows that every absolutely summable sequence 
belongs to the Bourgain--Morrey scale, and the corresponding norm estimate 
quantifies this embedding. 
This property will later ensure that convolution with $\ell^1$ kernels 
acts continuously on $\ell^{p}_{q,r}$.

\begin{proposition}\label{contencion}
Let $1\leq p<q<\infty$ and $1\leq r< \infty$. If $y\in \ell^{1}(\mathbb{Z})$, then $y\in \ell^{p}_{q,r}(\mathbb{Z})$ and
\begin{equation}\label{ineql1}
    \left\Vert  y\right\Vert _{\ell^{p}_{q,r}}\leq K\left\Vert y\right\Vert
_{\ell^{1}},
\end{equation}
for some $K>0$.
\end{proposition}

\begin{proof}
    Let $y\in \ell^{1}$ be arbitrary.  Since $e^{0}\in \ell^{p}_{q,r}$, by Lemma \ref{lem:trans_conv} (ii),
    it follows that
    \begin{equation*}
       \left\Vert e^{0}\ast y\right\Vert _{\ell^{p}_{q,r}}\leq\left\Vert y\right\Vert
_{\ell^{1}}\left\Vert e^{0}\right\Vert _{\ell^{p}_{q,r}} ,
    \end{equation*}
but observing that
\begin{equation*}
    (e^{0}\ast y)(k):=\sum_{\bar{m} \in \mathbb{Z}}e^{0}(\bar{m})y(k-\bar{m})=y(k),
\end{equation*}
we conclude 
 \begin{equation*}
       \left\Vert y\right\Vert _{\ell^{p}_{q,r}}=\left\Vert e^{0}\ast y\right\Vert _{\ell^{p}_{q,r}}\leq\left\Vert y\right\Vert
_{\ell^{1}}\left\Vert e^{0}\right\Vert _{\ell^{p}_{q,r}} ,
    \end{equation*}
this implies  $y\in \ell^{p}_{q,r}$ and taking $K=\left\Vert e^{0}\right\Vert _{\ell^{p}_{q,r}}$, we obtain (\ref{ineql1}).
\end{proof}
\begin{remark}
The estimate \eqref{ineql1} expresses the continuous embedding 
$\ell^{1}(\mathbb{Z})\hookrightarrow \ell^{p}_{q,r}(\mathbb{Z})$. 
Although elementary, it plays an important role: 
it guarantees that $\ell^{p}_{q,r}$ inherits many 
convolution and approximation properties of $\ell^{1}$. 
In particular, this inclusion implies that $\ell^{1}$ 
is dense in $\ell^{p}_{q,r}$.
\end{remark}

\begin{corollary}
\label{cor:density-L1}
For \(1 \leq p < q < \infty\) and \(1 \leq r < \infty\), the space \(\ell^1\) is dense in \(\ell^p_{q,r}\) with respect to the norm \(\|\cdot\|_{\ell^p_{q,r}}\).
\end{corollary}
\begin{proof}
Since \(c_{00}\) is dense in \(\ell^p_{q,r}\) by Theorem~\ref{Density} and \(c_{00} \subset \ell^1\), it follows that every element of \(\ell^p_{q,r}\) can be approximated by sequences from \(c_{00}\), which are contained in \(\ell^1\). Hence, \(\ell^1\) is dense in \(\ell^p_{q,r}\).
\end{proof}
   
The next statement complements Proposition~\ref{contencion} 
by showing that $\ell^{p}_{q,r}$ is continuously embedded into the classical space $\ell^{r}$. 
Together with the previous inclusion $\ell^{1}\hookrightarrow\ell^{p}_{q,r}$,
this places $\ell^{p}_{q,r}$ between $\ell^{1}$ and $\ell^{r}$ in the natural hierarchy of sequence spaces. 

\begin{proposition}\label{contentiononlr}
       Let $1\leq p<q<\infty$ and $1\leq r< \infty$. If $x\in \ell^{p}_{q,r}(\mathbb{Z})$, then $x\in \ell^{r}(\mathbb{Z})$ and 
       \begin{equation}
             \|x\|_{\ell^{r}(\mathbb{Z})}\leq \| x\| _{\ell^{p}_{q,r}(\mathbb{Z})}.
       \end{equation}
\end{proposition}
    \begin{proof}
 Note that for $j=0$, the dyadic intervals are singletons, i.e., $I(0,k)=\{k\}$, and hence $\vert I(0,k)\vert =1$, then
    \begin{equation*}
        \begin{split}
          \displaystyle     \|x\|_{\ell^{r}(\mathbb{Z})}&=\left[
\sum\limits_{k\in\mathbb{Z}}
\left\vert x(k)\right\vert ^{r}\right]  ^{1/r}\\
&=\left[\sum\limits_{k\in\mathbb{Z}}
\left\vert I\left(  0,k\right)  \right\vert ^{\frac{r}{q}-\frac{r}
{p}}\left(
{\displaystyle\sum\limits_{l\in I\left(  0,k\right)  }}
\left\vert x(l)\right\vert ^{p}\right)  ^{r/p}\right]  ^{1/r}\\
&\leq \left[
\sum_{j = 0}^{\infty} \sum_{k \in \mathbb{Z}}
\left\vert I\left(  j,k\right)  \right\vert ^{\frac{r}{q}-\frac{r}
{p}}\left(
{\displaystyle\sum\limits_{l\in I\left(  j,k\right)  }}
\left\vert x(l)\right\vert ^{p}\right)  ^{r/p}\right]  ^{1/r}=\| x\| _{\ell^{p}_{q,r}(\mathbb{Z})}.
        \end{split}
    \end{equation*}
This completes the proof.
    \end{proof}
 \begin{remark}\label{rem:l1_case}
From Propositions~\ref{contencion} and~\ref{contentiononlr} we have the continuous embeddings
\[
\ell^{1}(\mathbb Z)\hookrightarrow \ell^{p}_{q,r}(\mathbb Z)\hookrightarrow \ell^{r}(\mathbb Z),
\qquad 1\le p<q<\infty,\ 1\le r<\infty.
\]
In particular, for $r=1$ we obtain $\ell^{p}_{q,1}(\mathbb Z)=\ell^{1}(\mathbb Z)$ as sets, and the norms
$\|\cdot\|_{\ell^{p}_{q,1}}$ and $\|\cdot\|_{\ell^{1}}$ are equivalent.
\end{remark}

The following corollary records a simple but useful consequence of the embedding
$\ell^{p}_{q,r}\hookrightarrow \ell^{r}$: the space $\ell^{p}_{q,r}$ is dense in $\ell^{r}$
for $r<\infty$. This will be convenient when approximating elements of $\ell^{r}$ by
finitely supported sequences.

\begin{corollary}[Density of $\ell^{p}_{q,r}$ in $\ell^{r}$]
\label{cor:lpqr-dense-in-lr}
Let \(1 \le p < q < \infty\) and \(1 \le r < \infty\). Then \(\ell^{p}_{q,r}\) is dense in \(\ell^{r}\) with respect to the \(\ell^{r}\)-norm.
\end{corollary}

\begin{proof}
Since $c_{00}$ is dense in $\ell^{r}$ for $1\le r<\infty$ and
\[
c_{00}\subset \ell^{1}\subset \ell^{p}_{q,r}\subset \ell^{r},
\]
it follows that $\ell^{p}_{q,r}$ contains a dense subset of $\ell^{r}$, hence it is dense in $\ell^{r}$.
\end{proof}

\medskip

We also note the endpoint $r=1$. In this case the Bourgain--Morrey space does not
introduce a new class: one has $\ell^{p}_{q,1}=\ell^{1}$ with equivalent norms
(see Remark~\ref{rem:l1_case}). Consequently, the dual space agrees with the classical one.

\begin{corollary}[Duality in the endpoint $r=1$]\label{cor:endpoint-duality}
For every $1\le p<q<\infty$ one has $\ell^{p}_{q,1}=\ell^{1}$ with equivalence of norms. In particular,
\[
(\ell^{p}_{q,1})^* = (\ell^1)^* = \ell^\infty.
\]
\end{corollary}

The next example shows that when $r>1$ we have $\ell^{1}\subsetneq \ell^{p}_{q,r}$.
   \begin{example}
Let $1\le p<q<\infty$ and $1<r<\infty$. Consider the sequence
\[
x(k)=
\begin{cases}
k^{-1}, & k\ge 1,\\
0, & \text{otherwise}.
\end{cases}
\]
Then $x\in \ell^{p}_{q,r}\setminus \ell^{1}$.
\begin{proof}
Clearly $x\notin \ell^1$ since $\sum_{k\ge1}k^{-1}=+\infty$.
Assume first that $p>1$. For $k=0$, by the integral test,
\[
\sum_{l\in I(j,0)}|x(l)|^{p}=\sum_{l=1}^{2^{j}-1}l^{-p}
\le 1+\int_{1}^{2^{j}-1}t^{-p}\,dt
\le \frac{p}{p-1},
\]
hence
\begin{equation}\label{estimate1}
\Big( \sum_{l\in I(j,0)}|x(l)|^{p}\Big)^{r/p}\le \Big(\frac{p}{p-1}\Big)^{r/p}.
\end{equation}
For $k\ge1$,
\[
\sum_{l\in I(j,k)}|x(l)|^{p}
=\sum_{l=2^{j}k}^{2^{j}(k+1)-1}l^{-p}
\le 2^{j(1-p)}k^{-p},
\]
and therefore
\begin{equation}\label{estimate2}
\sum_{k\neq 0}\Big(\sum_{l\in I(j,k)}|x(l)|^{p}\Big)^{r/p}
\le \zeta(r)\,2^{jr(\frac1p-1)}.
\end{equation}
Combining \eqref{estimate1}--\eqref{estimate2} gives $\|x\|_{\ell^{p}_{q,r}}<\infty$.

For the endpoint $p=1$, the estimate \eqref{estimate2} is unchanged, while
\[
\sum_{l\in I(j,0)}|x(l)|
=\sum_{l=1}^{2^{j}-1}\frac1l
\le 1+\int_{1}^{2^{j}}\frac{dt}{t}
=1+j\log 2.
\]
Thus the $k=0$ contribution is bounded by a constant multiple of
$2^{j(\frac rq-r)}(1+j)^r$, which is summable in $j$ since $q>1$; hence
$\|x\|_{\ell^{1}_{q,r}}<\infty$ as well.
\end{proof}
\end{example}

The next example shows that when $r>q$ we have $ \ell^{p}_{q,r}\subsetneq\ell^{r}$.
  
\begin{example}\label{ex:strict_inclusion_r_greater_q}
Let $1\le p<q<\infty$ and $1<r<\infty$ be fixed. Let $s\ge 1$ and define
\[
x_{s}(k)=
\begin{cases}
|k|^{-1/s}, & k\neq0,\\[4pt]
0, & k=0.
\end{cases}
\]
Assume that $r>q$ and choose $s$ such that $q\le s<r$. Then $x_{s}\in\ell^{r}$ but $x_{s}\notin\ell^{p}_{q,r}$.
\end{example}

\begin{proof}
Since $s<r$, we have $x_s\in \ell^{r}$.

We prove that $x_s\notin \ell^{p}_{q,r}$. Using the dyadic decomposition, for each scale $j\ge0$ consider the dyadic cell
\[
I(j,1)=\{2^{j},2^{j}+1,\dots,2^{j+1}-1\}.
\]
For $l\in I(j,1)$ we have $2^{j}\le l\le 2^{j+1}-1$, hence
\[
\sum_{l\in I(j,1)}|x_{s}(l)|^{p}
=\sum_{l=2^{j}}^{2^{j+1}-1} l^{-p/s}
\ge 2^{j}\,(2^{j+1})^{-p/s}
=2^{-p/s}\,2^{j(1-p/s)}.
\]
Therefore
\[
\Big(\sum_{l\in I(j,1)}|x_{s}(l)|^{p}\Big)^{r/p}
\ge 2^{-r/s}\,2^{j r\big(\frac{1}{p}-\frac{1}{s}\big)}.
\]
Multiplying by the outer weight $(2^{j})^{\frac{r}{q}-\frac{r}{p}}$ from the definition of the $\ell^{p}_{q,r}$-norm yields the contribution
\[
(2^{j})^{\frac{r}{q}-\frac{r}{p}}
\Big(\sum_{l\in I(j,1)}|x_{s}(l)|^{p}\Big)^{r/p}
\ge C\,2^{j r\big(\frac{1}{q}-\frac{1}{s}\big)},
\qquad C:=2^{-r/s}>0.
\]
Summing over $j\ge0$ gives the lower bound
\[
\|x_{s}\|_{\ell^{p}_{q,r}}^{r}
\ge C\sum_{j\ge0}2^{j r(\frac{1}{q}-\frac{1}{s})}.
\]
Since $s\ge q$, we have $\frac{1}{q}-\frac{1}{s}\ge0$, and hence the geometric series on the right diverges. Therefore
$\|x_{s}\|_{\ell^{p}_{q,r}}=\infty$, i.e. $x_{s}\notin\ell^{p}_{q,r}$.
\end{proof}

The two examples before show that in particular, when $r>q$, we obtain the proper inclusions  $\ell^{1}\subsetneq \ell^{p}_{q,r}\subsetneq \ell^{r}$.

\begin{corollary}[Dual inclusions]
Let $1\le p<q<\infty,\; 1\leq r<\infty$. Then there are continuous embeddings
\[
\ell^{r'} \hookrightarrow (\ell^p_{q,r})^{*} \hookrightarrow \ell^\infty.
\]
\end{corollary}

\begin{remark}
The previous results shows that \(\ell^{r'}\) sits naturally inside \((\ell^p_{q,r})^{*}\) and that every bounded functional on \(\ell^p_{q,r}\) is represented on \(c_{00}\) by a bounded coordinate sequence. In many situations one can study \((\ell^p_{q,r})^{*}\) further using these two anchors (the subspace \(\ell^{r'}\) and the ambient \(\ell^\infty\)).
\end{remark}

\section{Equivalent norms on classical $\ell^p$ spaces via discrete Bourgain–Morrey spaces}

\label{sec:equivalent-norms}

In this section, we establish one of the most remarkable features of discrete Bourgain-Morrey spaces: their ability to generate uncountably many equivalent norms on classical sequence spaces. Specifically, we prove that for any $1 \leq p < \infty$ and $q > p$, the space $\ell^{p}_{q,p}$ coincides exactly with $\ell^p$ with equivalent norms. This identification yields a continuum of distinct but equivalent norms on each $\ell^p$ space, parameterized by the Morrey parameter $q > p$.

This discovery provides not only a new perspective on classical spaces but also a powerful tool for potential applications in functional analysis, interpolation theory, and numerical analysis. The different norms capture distinct aspects of a sequence's behavior, balancing local and global properties in varying proportions according to the parameter $q$.

These findings significantly expand our understanding of norm equivalence on classical spaces and open new avenues for investigating functional analytic properties through the lens of Bourgain-Morrey type spaces.

\begin{proof}[Proof of Theorem \ref{equivalent norms}]
We start from the definition of the norm in the discrete Bourgain-Morrey space:
\begin{equation*}
     \begin{split}
     \| x\| _{\ell_{q,p}^{p}(\mathbb{Z})}&=   \left[\sum\limits_{k\in\mathbb{Z},\,j\in\mathbb{N}_{0}}
\left\vert I\left(  j,k\right)  \right\vert ^{\frac{p}{q}-\frac{p}{p}}\left(
\sum\limits_{l\in I\left(  j,k\right)  }
\left\vert x(l)\right\vert ^{p}\right)  ^{p/p}\right]  ^{1/p}\\
&= \left[\sum\limits_{j\in\mathbb{N}_{0}}\sum_{k\in\mathbb{Z}}
\left\vert I\left(  j,k\right)  \right\vert ^{\frac{p}{q}-1}\left(
\sum\limits_{l\in I\left(  j,k\right)  }
\left\vert x(l)\right\vert ^{p}\right) \right]  ^{1/p}.
     \end{split}
\end{equation*}

Since $\left| I(j,k)\right| = 2^{j}$ for every $j\in\mathbb{N}_{0}$ and $k\in\mathbb{Z}$, we obtain
\begin{equation*}
     \begin{split}
&=\left[\sum\limits_{j\in\mathbb{N}_{0}}
\left(2^{j}\right) ^{\frac{p}{q}-1}\sum_{k\in\mathbb{Z}}\left(
\sum\limits_{l\in I\left(  j,k\right)  }
\left\vert x(l)\right\vert ^{p}\right) \right]  ^{1/p}.
     \end{split}
\end{equation*}

Now, by reorganizing the inner sum over $l\in I(j,k)$ and pulling out the weight depending only on $j$, we obtain
\begin{equation*}
     \begin{split}
&=\left[\sum\limits_{j\in\mathbb{N}_{0}}
\left(2 ^{\frac{p}{q}-1}\right)^{j}\sum_{k\in\mathbb{Z}}
\sum\limits_{l\in I\left(  j,k\right)  }
\left\vert x(l)\right\vert ^{p} \right]  ^{1/p}.
     \end{split}
\end{equation*}

Since the family $\{I(j,k):k\in\mathbb{Z}\}$ forms a partition of $\mathbb{Z}$, the double sum in $k$ and $l$ simply reduces to the sum over all indices $k\in\mathbb{Z}$:
\begin{equation*}
     \begin{split}
&=\left[\sum\limits_{j\in\mathbb{N}_{0}}
\left(2 ^{\frac{p}{q}-1}\right)^{j}\sum_{k\in\mathbb{Z}}
\left\vert x_{k}\right\vert ^{p} \right]  ^{1/p}.
     \end{split}
\end{equation*}

Hence, factoring out the $\ell^{p}$-norm of $x$, we arrive at
\begin{equation*}
     \begin{split}
&=\left[\sum\limits_{j\in\mathbb{N}_{0}}
\left(2 ^{\frac{p}{q}-1}\right)^{j}\right]  ^{1/p}\|x\|_{\ell^{p}(\mathbb{Z})}.
     \end{split}
\end{equation*}

Finally, since $\tfrac{p}{q}-1<0$, the geometric series converges and defines a finite constant $C_{p,q}>0$. Therefore,
\begin{equation}\label{multconstnorm}
\|x\| _{\ell_{q,p}^{p}(\mathbb{Z})}=C_{p,q}\|x\|_{\ell^{p}(\mathbb{Z})},
\end{equation}
which proves the equivalence of the two norms.
\end{proof}

\begin{corollary}[The case $r=p$]\label{rmk:case_q_infty}
For $1\le p<\infty$ and $q>p$, we have $\ell^p_{q,p} = \ell^p$ with equivalent norms. Consequently:
\begin{enumerate}
\item $\ell^p_{q,p}$ inherits all Banach space properties of $\ell^p$ (completeness, separability, Schauder basis).
\item If $1<p<\infty$, then $\ell^p_{q,p}$ is reflexive and its dual is (isomorphic to) $\ell^{p'}$.
\item An operator $T$ is bounded on $\ell^p$ if and only if it is bounded on $\ell^p_{q,p}$, with equivalent operator norms.
\end{enumerate}
\end{corollary}

\begin{corollary}[Three uncountable families of equivalent norms on $\ell^p$]
\label{cor:three-families-lp}
Let $1 \leq p < \infty$ and fix $q > p$. For $r = p$, the Bourgain-Morrey space $\ell^p_{q,p}(\mathbb{Z})$ coincides with the classical $\ell^p(\mathbb{Z})$ space, and the following three families of norms (parameterized by $q > p$) are equivalent to the standard $\ell^p$-norm:

\begin{enumerate}
    \item[(a)] \textbf{Centered-interval norms:}
    \[
    \|x\|^{\mathcal{S}}_{\ell_{q,p}^{p}} = \left( \sum_{N = 0}^{\infty} \sum_{m \in \mathbb{Z}} (2N+1)^{p\left(\frac{1}{q} - \frac{1}{p}\right)} \left( \sum_{l \in S_{m,N}} |x(l)|^{p} \right) \right)^{1/p}.
    \]

    \item[(b)] \textbf{Dyadic-interval norms:}
    \[
    \|x\|^{\mathcal{D}}_{\ell_{q,p}^{p}} = \left( \sum_{j = 0}^{\infty} \sum_{k \in \mathbb{Z}} (2^{j})^{p\left(\frac{1}{q} - \frac{1}{p}\right)} \left( \sum_{l \in I(j,k)} |x(l)|^{p} \right) \right)^{1/p}.
    \]

    \item[(c)] \textbf{Dyadic-length-interval norms:}
    \[
    \|x\|^{\mathcal{A}}_{\ell_{q,p}^{p}} = \left( \sum_{N = 0}^{\infty} \sum_{m \in \mathbb{Z}} (2^{N+1}+1)^{p\left(\frac{1}{q} - \frac{1}{p}\right)} \left( \sum_{l \in S_{m,2^{N}}} |x(l)|^{p} \right) \right)^{1/p}.
    \]
\end{enumerate}

In particular, for each fixed $q > p$, the three norms above are equivalent to each other and to the classical $\ell^p$-norm. Moreover, as $q$ varies over $(p, \infty)$, each family $\{\|\cdot\|^{\mathcal{S}}_{\ell_{q,p}^{p}}\}_{q>p}$, $\{\|\cdot\|^{\mathcal{D}}_{\ell_{q,p}^{p}}\}_{q>p}$, and $\{\|\cdot\|^{\mathcal{A}}_{\ell_{q,p}^{p}}\}_{q>p}$ forms an uncountable family of equivalent norms on $\ell^p(\mathbb{Z})$.
\end{corollary}
\begin{remark}[On the potential novelty of the centered-family norm]
\label{rem:S-novelty}
While the dyadic-family norm $\|\cdot\|_{\ell_{q,p}^{p}}^{\mathcal{D}}$ reduces to a constant multiple of the classical $\ell^p$-norm (as shown in (\ref{multconstnorm})), the situation for the centered-family norm $\|\cdot\|_{\ell_{q,p}^{p}}^{\mathcal{S}}$ is less clear.

This norm aggregates information from overlapping intervals of varying sizes in a non-trivial way. While it is equivalent to $\|\cdot\|_{\ell^p}$, it is not immediately obvious whether it is merely a scalar multiple of the classical norm. If it turns out that $\|\cdot\|_{\ell_{q,p}^{p}}^{\mathcal{S}}$ is not a simple multiple of $\|\cdot\|_{\ell^p}$, then it would provide a genuinely distinct functional that could capture both local and global features of sequences in new ways.

This potential structural richness could make $\|\cdot\|_{\ell_{q,p}^{p}}^{\mathcal{S}}$ a more informative measure than the standard $\ell^p$-norm, while still inducing the same topology. Such a result could be a powerful tool for improving estimates in inequalities for $\ell^p$ spaces. One could analyze particular values of $q$ to optimize the norm $\|\cdot\|_{\ell_{q,p}^{p}}^{\mathcal{S}}$ in specific contexts. The existence of uncountably many equivalent norms on $\ell^p$ suggests several natural questions. For instance, it would be interesting to investigate whether these alternative norms may lead to sharper estimates for certain classical operators, or whether they can provide new insights into the geometry of $\ell^p$ spaces. We do not pursue these directions here, but we believe that such problems deserve further attention.
\end{remark}

\begin{remark}[The Hilbertian case $p = r = 2$]
\label{rem:hilbertian_case}
In the case $p = r = 2$, the Bourgain-Morrey sequence space $\ell^{2}_{q,2}$ is isomorphic to the classical Hilbert space $\ell^2$ when equipped with the dyadic-family norm $\|\cdot\|^{\mathcal{D}}_{\ell_{q,2}^{2}}$, as this norm reduces to a constant multiple of the standard $\ell^2$-norm. Consequently, the space admits a natural Hilbertian structure inherited from $\ell^2$.

However, the situation is more subtle for the centered-family norm $\|\cdot\|^{\mathcal{S}}_{\ell_{q,2}^{2}}$ and the dyadic-length-family norm $\|\cdot\|^{\mathcal{A}}_{\ell_{q,2}^{2}}$. While these norms are equivalent to the $\ell^2$-norm (and thus induce the same topology and completeness), it is not immediately clear whether they arise from an inner product. 

The centered-family norm, in particular, aggregates information over overlapping intervals of varying sizes, which may break the parallelogram law required for a Hilbert space structure. Similarly, the dyadic-length-family norm, though based on a subfamily of centered intervals, may also fail to satisfy the parallelogram identity due to the non-orthogonal nature of the projections onto these intervals.

Thus, while $\ell^{2}_{q,2}$ is isomorphic to $\ell^2$ as a Banach space for all three families, only the dyadic-family norm is guaranteed to be Hilbertian. The centered and dyadic-length families provide equivalent renormings of $\ell^2$ whose Hilbertian character remains an open question worthy of further investigation.
\end{remark}
\begin{proposition}[The case $q=\infty$]\label{prop:q_infty}
Let $1\le p<\infty$ and $1\le r<\infty$. With the convention $1/\infty=0$, the Bourgain--Morrey space
$\ell^{p}_{\infty,r}(\mathbb{Z})$ coincides with $\ell^{r}(\mathbb{Z})$ up to equivalence of norms. More precisely,
\[
\|x\|_{\ell^r}\le \|x\|_{\ell^{p}_{\infty,r}} \le C_{p,r}\,\|x\|_{\ell^r}
\qquad (x\in\mathbb{R}^{\mathbb{Z}}),
\]
where one may take
\[
C_{p,r}=
\begin{cases}
2^{1/r}, & r\ge p,\\[2pt]
\big(1-2^{-r/p}\big)^{-1/r}, & r<p.
\end{cases}
\]
\end{proposition}

\begin{proof}
Since $1/\infty=0$, the dyadic norm reads
\[
\|x\|_{\ell^{p}_{\infty,r}}^{r}
=\sum_{j\ge0}\sum_{k\in\mathbb Z} 2^{-jr/p}\Big(\sum_{l\in I(j,k)}|x(l)|^{p}\Big)^{r/p}.
\]

For the lower bound, take $j=0$ (so $I(0,k)=\{k\}$):
\[
\|x\|_{\ell^{p}_{\infty,r}}^{r} \ge \sum_{k\in\mathbb Z} |x(k)|^{r}=\|x\|_{\ell^r}^{r}.
\]

For the upper bound, fix a dyadic interval $I=I(j,k)$ with $|I|=2^j$.

If $r\ge p$, then $\|x\|_{\ell^p(I)}\le |I|^{\,1/p-1/r}\|x\|_{\ell^r(I)}$, hence
\[
2^{-jr/p}\|x\|_{\ell^p(I)}^{r}
\le 2^{-jr/p}\,|I|^{r(1/p-1/r)}\|x\|_{\ell^r(I)}^{r}
=2^{-j}\sum_{l\in I}|x(l)|^{r}.
\]
Summing over $k$ at fixed $j$ gives
\[
\sum_{k}2^{-jr/p}\|x\|_{\ell^p(I(j,k))}^{r}
\le 2^{-j}\sum_{l\in\mathbb Z}|x(l)|^{r},
\]
and summing over $j\ge0$ yields
\[
\|x\|_{\ell^{p}_{\infty,r}}^{r}\le2\,\|x\|_{\ell^r}^{r}.
\]

If $r<p$, then $\|x\|_{\ell^p(I)}\le \|x\|_{\ell^r(I)}$, hence
\[
2^{-jr/p}\|x\|_{\ell^p(I)}^{r}\le 2^{-jr/p}\sum_{l\in I}|x(l)|^{r}.
\]
Summing over $k$ at fixed $j$ gives
\[
\sum_{k}2^{-jr/p}\|x\|_{\ell^p(I(j,k))}^{r}\le 2^{-jr/p}\sum_{l\in\mathbb Z}|x(l)|^{r},
\]
and summing over $j\ge0$ yields
\[
\|x\|_{\ell^{p}_{\infty,r}}^{r}\le \Big(\sum_{j\ge0}2^{-jr/p}\Big)\|x\|_{\ell^r}^{r}
=\frac{1}{1-2^{-r/p}}\,\|x\|_{\ell^r}^{r}.
\]
This completes the proof.
\end{proof}

\begin{remark}[The boundary case $q=\infty$]\label{rem:q_infty}
With the convention $1/\infty=0$, the space $\ell^{p}_{\infty,r}(\mathbb{Z})$ is isomorphic to $\ell^{r}(\mathbb{Z})$
(up to equivalence of norms), independently of $p$. Consequently,
\[
(\ell^{p}_{\infty,r})^*\ \simeq\ (\ell^{r})^*=\ell^{r'},
\]
and in particular $\ell^{p}_{\infty,r}$ is reflexive whenever $1<r<\infty$.
\end{remark}


\section{Duality of the spaces \(\ell^{p}_{q,r}\)}

The study of duality for function spaces is a cornerstone of functional analysis, providing deep insights into the geometric structure, interpolation properties, and operator theory associated with these spaces. For sequence spaces, a precise description of the dual space not only reveals the underlying topology but also facilitates applications in harmonic analysis, spectral theory, and numerical analysis. In this section, we address the duality theory for the discrete (sequences) Bourgain–Morrey spaces \(\ell^{p}_{q,r}(\mathbb{Z})\), thereby completing the structural analysis initiated in [11] and further developed in the preceding sections of this work.

Our motivation stems from two primary sources. On the one hand, in the continuous setting, the seminal works of Masaki and Hatano et al. \cite{masaki2016two, hatano2004bourgain}, as well as the recent study of Bai, Guo and Xu~\cite{preduals of Banach space valued}, established duality results for Bourgain–Morrey-type function spaces by introducing preduals defined via suitably adapted block decompositions. On the other hand, in the discrete realm, it is natural to seek discrete analogues of known continuous results, as has been successfully done for Lorentz, Orlicz, and Besov sequence spaces. In particular, the construction of preduals via atomic or block decompositions has proven effective in various discrete contexts, including classical discrete Morrey spaces \cite{MorreySeq2020}.

In this section, we show that the dual of \(\ell^{p}_{q,r}(\mathbb{Z})\) can be isometrically identified with a block space \(\mathrm{h}^{p'}_{q',r'}(\mathbb{Z})\), defined through a decomposition into \((q',p')\)-blocks adapted to the dyadic intervals. The proofs are discrete adaptations of the techniques introduced by Hatano and Masaki, though additional care is required due to the discrete setting and the absence of arbitrarily small dyadic intervals, which limits pointwise approximations. This result not only enriches the theory of \(\ell^{p}_{q,r}\) spaces but also paves the way for applications in discrete harmonic analysis, operator theory, and interpolation.

The section is organized as follows. In Subsection 6.1, we define the block space
$\mathrm{h}^{p'}_{q',r'}(\mathbb{Z})$ through decompositions into $(q',p')$--blocks adapted
to dyadic intervals and record its basic properties, including the embedding
$\ell^{r'}\hookrightarrow \mathrm{h}^{p'}_{q',r'}$.
Subsection 6.2 contains the duality theory: Theorem~\ref{the:preduality} identifies
$(\mathrm{h}^{p'}_{q',r'})^*$ isometrically with $\ell^{p}_{q,r}$, while
Theorem~\ref{thm:duality} gives the converse identification $(\ell^{p}_{q,r})^*\cong
\mathrm{h}^{p'}_{q',r'}$.
As consequences we obtain reflexivity of $\ell^{p}_{q,r}$ (Proposition~\ref{prop:reflexivity-lpqr})
and, in the special case $r=p$, the identification $\mathrm{h}^{p'}_{q',p'}\simeq \ell^{p'}$
(Corollary~\ref{reflex:blockr=p}).
Finally, in the Appendix we prove the key combinatorial estimate (Lemma~\ref{lem:key-appendix})
used to control the $\ell^{p}_{q,r}$--norm of block series and to complete the duality arguments.

\subsection{A predual of $\ell^{p}_{q,r}$: The block space} 
  Throughout we assume $1<p<q<\infty$ and $1<r<\infty$, and let $p',q',r'$ denote the conjugate exponents.

  \begin{definition} 
Let $1<p< q< \infty$. We say that a sequence $y$ on $\mathbb{Z}$ is a \emph{$(q',p')$-block} with respect to a dyadic interval $I\in \mathcal{D}$ if $\operatorname{supp}y\subset I$ and 
\begin{equation*} 
\|y\|_{\ell^{p'}(\mathbb{Z})}= \|y\|_{\ell^{p'}(I)} \leq |I|^{\frac{1}{p'}-\frac{1}{q'}}. 
\end{equation*}
A sequence $y$ is simply called a $(q',p')$-block if there exists a dyadic interval $I_{0}$ such that $y$ is a $(q',p')$-block with respect to $I_{0}$. 
\end{definition} 

\begin{definition} 
Let $1< p< q<\infty $ and $1<r< \infty$. We define the block sequence space $\mathrm{h}_{q',r'}^{p'}(\mathbb{Z})$ as the set of all sequences $x\in \mathbb{R}^{\mathbb{Z}}$ which can be represented as 
\begin{equation}\label{representation} 
x=\sum_{(j,k)\in \mathbb{N}_{0}\times\mathbb{Z}} \lambda_{k}^{j}a_{k}^{j}, 
\end{equation} 
with $\lambda=(\lambda_{k}^{j})_{(j,k)\in \mathbb{N}_{0}\times\mathbb{Z}} \in \ell^{r'}(\mathbb{N}_{0}\times\mathbb{Z})$ and $a_{k}^{j}$ a $(q',p')$-block supported on $I(j,k)$ for each $(j,k)\in \mathbb{N}_{0}\times\mathbb{Z}$. The convergence of \eqref{representation} is understood in $\ell^{1}_{\mathrm{loc}}$. The norm $\|x\|_{\mathrm{h}_{q',r'}^{p'}(\mathbb{Z})}$ is defined as 
\begin{equation}
\|x\|_{\mathrm{h}_{q',r'}^{p'}(\mathbb{Z})} := \inf_{\lambda}\|\lambda\|_{\ell^{r'}}, 
\end{equation} 
where the infimum is taken over all admissible representations \eqref{representation}, with $a_{k}^{j}$ a $(q',p')$-block supported in $I(j,k)$ for each $(j,k)$. 
\end{definition} 
\begin{remark}
Let $1< p<q<\infty$ and $1\le r<\infty$. Then $\ell^{r'}(\mathbb Z)\subset \mathrm{h}^{p'}_{q',r'}(\mathbb Z)$.
\begin{proof}
Let $x\in\ell^{r'}$. Use the admissible representation
\[
x=\sum_{(j,k)\in\mathbb N_0\times\mathbb Z}\lambda_k^j B_k^j,
\qquad 
\lambda_k^0=x(k),\ \lambda_k^j=0\ (j\ge1),\ B_k^j=\mathbf 1_{I(j,k)}.
\]
Hence
\[
\|x\|_{\mathrm{h}^{p'}_{q',r'}}\le \|\lambda\|_{\ell^{r'}(\mathbb N_0\times\mathbb Z)}=\|x\|_{\ell^{r'}(\mathbb Z)}.
\]
\end{proof}
\end{remark}

\begin{theorem}[Pre-duality] \label{the:preduality}
Let $1<p< q<\infty$ and $1<r<\infty$. Then the map \[ x\in \ell_{q,r}^p(\mathbb{Z})\longmapsto L_x\in \big(\mathrm{h}_{q',r'}^{p'}(\mathbb{Z})\big)^*, \qquad L_x(y)=\sum_{m\in\mathbb{ Z}} x(m)y(m), \] is a linear isometric isomorphism. In particular, every bounded linear functional on $\mathrm{h}_{q',r'}^{p'}(\mathbb{Z})$ is of the form $L_x$ for a unique $x\in\ell_{q,r}^p(\mathbb{Z})$, and $\|L_x\|=\|x\|_{\ell_{q,r}^p}$. 
\end{theorem} 
\begin{proof} 
\textbf{Step 1.} (\emph{Boundedness.}) Let $x\in\ell_{q,r}^p(\mathbb{Z})$ and $y\in \mathrm{h}_{q',r'}^{p'}(\mathbb{Z})$ with an admissible block decomposition
\[ y=\sum_{(j,k)\in \mathbb{N}_{0}\times\mathbb{Z}}\lambda_k^j a_k^j, \qquad (\lambda_k^j)\in \ell^{r'}(\mathbb N_0\times\mathbb{Z}). \] 
Each $a_k^j$ is a $(q',p')$-block supported in $I(j,k)$. Then, by Hölder’s inequality, we obtain
\[ \begin{aligned}
|L_x(y)| &= \Big|\sum_{m} x(m)\sum_{(j,k)\in \mathbb{N}_{0}\times\mathbb{Z}}\lambda_k^j a_k^j(m)\Big| \le \sum_{(j,k)\in \mathbb{N}_{0}\times\mathbb{Z}} |\lambda_k^j|\sum_{m\in I(j,k)} |x(m)||a_k^j(m)| \\ &\le \sum_{(j,k)\in \mathbb{N}_{0}\times\mathbb{Z}} |\lambda_k^j|\, \|x\|_{\ell^p(I(j,k))}\,\|a_k^j\|_{\ell^{p'}(I(j,k))}. 
\end{aligned} \]
By the block condition, $\|a_k^j\|_{\ell^{p'}(I(j,k))}\le |I(j,k)|^{1/q-1/p}$. Hence \[ |L_x(y)| \le \sum_{(j,k)\in \mathbb{N}_{0}\times\mathbb{Z}} |\lambda_k^j|\,|I(j,k)|^{1/q-1/p}\,\|x\|_{\ell^p(I(j,k))}. \] Applying Hölder in the $(j,k)$–index with exponents $(r',r)$ yields \[ |L_x(y)| \le \|\lambda\|_{\ell^{r'}} \Big(\sum_{(j,k)\in \mathbb{N}_{0}\times\mathbb{Z}} \big(|I(j,k)|^{1/q-1/p}\|x\|_{\ell^p(I(j,k))}\big)^r\Big)^{1/r}. \] The last factor is $\|x\|_{\ell_{q,r}^p}$, while $\|\lambda\|_{\ell^{r'}}$ is arbitrarily close to $\|y\|_{\mathrm{h}_{q',r'}^{p'}}$. Therefore \[ |L_x(y)| \le \|x\|_{\ell_{q,r}^p}\,\|y\|_{\mathrm{h}_{q',r'}^{p'}}. \] Thus $L_x$ is bounded with $\|L_x\|\le \|x\|_{\ell_{q,r}^p}$.\\ 
\textbf{Step 2.} \emph{(Representation of arbitrary functionals.)}
Let $f\in(\mathrm{h}^{p'}_{q',r'}(\mathbb Z))^*$ and fix a dyadic interval $I=I(j,k)$.
Consider the extension-by-zero operator $T_I:\ell^{p'}(I)\to \mathrm{h}^{p'}_{q',r'}(\mathbb Z)$.
By Lemma~\ref{lem:local_embedding} we have
\[
\|T_I z\|_{\mathrm{h}^{p'}_{q',r'}(\mathbb Z)}\le |I|^{\frac1p-\frac1q}\,\|z\|_{\ell^{p'}(I)}
\qquad (z\in \ell^{p'}(I)),
\]
hence $f\circ T_I$ is a bounded functional on $\ell^{p'}(I)$.
By finite-dimensional duality, there exists $x_I\in \ell^p(I)$ such that
\[
f(T_I z)=\sum_{m\in I} x_I(m)\,z(m)\qquad (z\in \ell^{p'}(I)).
\]
Moreover,
\[
\|x_I\|_{\ell^p(I)}=\|f\circ T_I\|_{(\ell^{p'}(I))^*}\le |I|^{\frac1p-\frac1q}\,\|f\|.
\]
If $I\subset J$ are dyadic, the above representation implies $x_J|_I=x_I$.
Thus we may define a global sequence $x\in \mathbb R^{\mathbb Z}$ by setting
$x(m):=x_{I(0,m)}(m)$, and then $f(y)=\sum_{m\in\mathbb Z}x(m)y(m)$ for all finitely supported $y$.

\textbf{Step 3.} (\emph{Norm control.}) For each dyadic interval $I$ define 
\[ b_I(m)= \begin{cases} |I|^{\tfrac{1}{q}-\tfrac{1}{p}}\,\|x\|_{\ell^p(I)}^{\,1-p}\,|x(m)|^{p-1}\operatorname{sgn}(x(m)), & \|x\|_{\ell^p(I)}>0, \\ 0,& \text{otherwise}. \end{cases} \] 
By direct computation, $b_I$ is a $(q',p')$–block. Testing $f$ against finite combinations $y_K=\sum_{I\in K} z_I b_I$ and applying Lemma~\ref{lem:testing_to_lrp}, we deduce that 
\[ \|x\|_{\ell_{q,r}^p} \le \|f\|. \] 
Thus $x\in\ell_{q,r}^p(\mathbb{Z})$ and $f=L_x$. \smallskip

\textbf{Step 4.} (\emph{Isometry and uniqueness.}) We have shown $\|L_x\|\le \|x\|_{\ell_{q,r}^p}$ and $\|x\|_{\ell_{q,r}^p}\le \|L_x\|$, so the identification is isometric. Uniqueness of $x$ follows by testing $f$ against unit vectors. 
\end{proof}

\subsection{Duality via block spaces} 
To maintain the flow of the argument, the proof of the following lemma is postponed to the appendix.
\begin{lemma}\label{lem:key-appendix}
Let $\{u_{I}\}_{I\in\mathcal D}$ be a family of sequences with $\supp  u_I\subset I$ and $\|u_I\|_{\ell^{p}(I)}\le1$ for every $I$. For any coefficient array $\lambda=(\lambda_I)_{I\in\mathcal D}\in\ell^{r}(\mathcal D)$ define
\[
x(n):=\sum_{I\in\mathcal D}\lambda_I\,|I|^{\,\frac{1}{q}-\frac{1}{p}}\,u_I(n),\qquad n\in\mathbb Z.
\]
Then there exists $C=C(p,q,r)>0$ such that
\[
\|x\|_{\ell^{p}_{q,r}}\le C\|\lambda\|_{\ell^{r}}.
\]
\end{lemma}
\begin{proposition}[Reflexivity of $\ell^{p}_{q,r}$]\label{prop:reflexivity-lpqr}
Let $1<p<q<\infty$ and $1<r<\infty$. Then $\ell^{p}_{q,r}(\mathbb Z)$ is reflexive.
\end{proposition}

\begin{proof}
Let $\mathcal D$ be the dyadic family. Consider the Banach space
\[
X:=\Big(\bigoplus_{I\in\mathcal D}\ell^p(I)\Big)_{\ell^r},
\qquad
\|(u_I)\|_X:=\Big(\sum_{I\in\mathcal D}\|u_I\|_{\ell^p(I)}^r\Big)^{1/r}.
\]
Since $1<p<\infty$ and $1<r<\infty$, the space $X$ is reflexive.
Define the linear map $J:\ell^{p}_{q,r}(\mathbb Z)\to X$ by
\[
(Jx)_I:=|I|^{\frac1q-\frac1p}\,x\mathbf 1_I\in \ell^p(I)\qquad(I\in\mathcal D).
\]
Then, by definition of the dyadic norm, $\|Jx\|_X=\|x\|_{\ell^{p}_{q,r}}$.
Hence $J$ is an isometric embedding and $\ell^{p}_{q,r}$ is isometric to a closed
subspace of the reflexive space $X$, therefore it is reflexive.
\end{proof}
\begin{theorem}[Duality]\label{thm:duality}
Let $1<p<q<\infty$ and $1<r<\infty$. Then
\[
(\ell^{p}_{q,r}(\mathbb Z))^{*}\ \cong\ \mathrm{h}^{p'}_{q',r'}(\mathbb Z)
\]
isometrically, via the canonical pairing $\langle x,y\rangle=\sum_{n\in\mathbb Z}x(n)y(n)$.
More precisely, for $f\in(\ell^{p}_{q,r})^*$ the sequence $y$ defined by $y(n):=f(e_n)$
belongs to $\mathrm{h}^{p'}_{q',r'}(\mathbb Z)$ and represents $f$ by
\[
f(x)=\sum_{n\in\mathbb Z}x(n)y(n)\qquad(x\in \ell^{p}_{q,r}).
\]
Conversely, every $y\in \mathrm{h}^{p'}_{q',r'}(\mathbb Z)$ defines a bounded functional on
$\ell^{p}_{q,r}$ by the same formula.
\end{theorem}

\begin{proof}
The converse direction is the same estimate as in Step~1 of Theorem~6.4:
if $y=\sum \lambda_I a_I$ is an admissible block representation with
$(\lambda_I)\in\ell^{r'}$ and $(q',p')$--blocks $a_I$, then Hölder on each block and then
on the $\ell^{r'}$--sum yields $|\langle x,y\rangle|\le \|x\|_{\ell^{p}_{q,r}}\|y\|_{\mathrm{h}^{p'}_{q',r'}}$.

For the identification of $(\ell^{p}_{q,r})^*$, combine Theorem~6.4 with reflexivity.
By Theorem~6.4 we have an isometric isomorphism
\[
\ell^{p}_{q,r}\ \cong\ (\mathrm{h}^{p'}_{q',r'}(\mathbb Z))^{*}.
\]
By Proposition~\ref{prop:reflexivity-lpqr}, $\ell^{p}_{q,r}$ is reflexive, hence
$(\mathrm{h}^{p'}_{q',r'})^{*}$ is reflexive and therefore $\mathrm{h}^{p'}_{q',r'}$ is reflexive.
Taking duals in the above identification gives
\[
(\ell^{p}_{q,r})^{*}\ \cong\ (\mathrm{h}^{p'}_{q',r'})^{**}
=\mathrm{h}^{p'}_{q',r'}.
\]
If $f\in(\ell^{p}_{q,r})^*$ corresponds to $y\in \mathrm{h}^{p'}_{q',r'}$ under this isomorphism,
then for every $n\in\mathbb Z$,
\[
y(n)=\langle e_n,y\rangle=f(e_n),
\]
so the representing sequence is exactly $y(n):=f(e_n)$, and the representation
$f(x)=\sum_n x(n)y(n)$ follows by linearity and density of $c_{00}$.
\end{proof}

\begin{remark}
The hypotheses $1<p<q<\infty$ and $1<r<\infty$ are essential for the duality identifications used above. Boundary cases (for instance $p=1$) are not covered by these theorems and reflexivity may fail there.
\end{remark}

\begin{corollary}[Identification of the block space when \(p = r\)]\label{reflex:blockr=p}
Let \(1 < p < q < \infty\) and let \(p' = p/(p-1)\) be the conjugate exponent of \(p\). 
When $r=p$, the block space $\mathrm{h}^{p'}_{q',p'}(\mathbb Z)$ coincides with $\ell^{p'}(\mathbb Z)$ as a set and the norms are equivalent.
That is, there exists a constant \(C > 0\) such that for every \(y \in \mathrm{h}^{p'}_{q', p'}(\mathbb{Z})\),
\[
C^{-1} \|y\|_{\ell^{p'}} \leq \|y\|_{\mathrm{h}^{p'}_{q', p'}} \leq C \|y\|_{\ell^{p'}}.
\]
In particular, the norms are equivalent and the spaces are isomorphic as Banach spaces.
\end{corollary}
\begin{proof}
The result follows immediately from Theorems 2.3 and 6.4. 
By Theorem 2.3, we have \(\ell^{p}_{q,p} = \ell^{p}\) with equivalent norms. 
By Theorem 6.4, the dual space of \(\ell^{p}_{q,p}\) is isometrically isomorphic to \(\mathrm{h}^{p'}_{q', p'}\). 
Since the dual of \(\ell^{p}\) is \(\ell^{p'}\), we conclude that \(\mathrm{h}^{p'}_{q', p'} \simeq \ell^{p'}\) with equivalent norms.
\end{proof}

\begin{remark}[Recovery of $(\ell^p)^*=\ell^{p'}$ and equivalent block norms]
As a consequence of Theorem \ref{the:preduality} and the duality result established in Theorem \ref{thm:duality}, we obtain that for fixed $1 < p < \infty$ and $q > p$, the classical space $\ell^{p'}$ admits a family of equivalent norms parameterized by $q > p$. Specifically, for each $q > p$, the norm
\[
\|y\|_{\mathrm{h}^{p'}_{q',p'}}
:=\inf\Big\{\|\lambda\|_{\ell^{p'}}:\ 
\begin{aligned}[t]
&y=\sum_{(j,k)\in\mathbb N_0\times\mathbb Z}\lambda_k^j a_k^j,\\
& a_k^j\ \text{is a $(q',p')$--block supported on } I(j,k)\Big\}.
\end{aligned}
\]

is equivalent to the standard $\ell^{p'}$-norm. In particular, since 
\[
(\mathrm{h}^{p'}_{p',p'})^* = \ell_{p,p}^p = \ell^p,
\]
we recover the classical duality $(\ell^p)^* = \ell^{p'}$ within this framework.

Moreover, as $q$ varies over $(p, \infty)$, we obtain an uncountable family of distinct but equivalent norms on $\ell^{p'}$. This provides a dual counterpart to the three families of equivalent norms on $\ell^p$ described in Corollary~5.0.2, further demonstrating the rich norm structure that arises from the Bourgain-Morrey framework.
\end{remark}

\begin{remark}[The endpoint case $r=\infty$ and discrete Morrey spaces]
When $r=\infty$, the Bourgain-Morrey sequence space $\ell_{q,\infty}^{p}$ 
coincides with the discrete Morrey space $\ell_q^p$ studied earlier by Haroske et al (cf.~\cite{MorreySeq2020}). 
In that work a block predual for $\ell_q^p$ was constructed (under a different notation), 
thus providing a duality framework at the endpoint $r=\infty$. 
Our main results for $1<p< q<\infty$ and $1<r<\infty$ can therefore be seen 
as a refinement and generalization of this picture: 
while the case $r=\infty$ is already known to admit a block predual, 
we establish the full duality between $\ell_{q,r}^p$ and the block space 
$\mathrm{h}_{q',r'}^{p'}$ throughout the interior range $1<r<\infty$. 
A natural open question that remains is to obtain a complete characterization of the dual 
of $\ell_{q,\infty}^p$ itself.
\end{remark}
The following table summarizes the duality theory of Bourgain-Morrey sequence spaces.

\begin{table}[H]
\centering
\renewcommand{\arraystretch}{1.5}
\begin{tabular}{|p{3cm}|p{3.5cm}|p{3.5cm}|p{3.5cm}|}
\hline
\textbf{Parameters} & \textbf{Space $\ell_{q,r}^p$} & \textbf{Predual space} & \textbf{Dual space} \\
\hline
$r=1$, $1\le p< q<\infty$ & $\ell_{q,1}^p = \ell^1$ (with equivalent norms, Cor. \ref{rem:l1_case}) & $c_0$: space of sequences converging to zero (see,\cite{albiac2006topics}) & $\ell^\infty$: space of bounded sequences \\
\hline
$1<p< q<\infty$, $1<r<\infty$, $r\neq p$ & $\ell_{q,r}^p$: proper Bourgain-Morrey sequence space & $\mathrm{h}^{p'}_{q',r'}$: block space defined via dyadic intervals & $(\ell_{q,r}^p)^* \cong \mathrm{h}^{p'}_{q',r'}$ (isometric isomorphism) \\
\hline
$1<p< q<\infty$, $r= p$ & $\ell_{q,p}^p = \ell^{p}$ (with equivalent norms, Theorem \ref{equivalent norms}) & $\ell^{p'}$: classical predual of $\ell^p$ (see \cite{folland1999real}) & $\ell^{p'}$: classical dual of $\ell^p$ (see \cite{folland1999real})\\
\hline
$q = \infty$, $1 \leq p < \infty$, $1 < r < \infty$ & $\ell_{\infty,r}^p = \ell^r$ (with equivalent norms, Remark \ref{rem:q_infty}) & $\ell^{r'}$: classical predual of $\ell^r$ (see \cite{folland1999real}) &  $\ell^{r'}$: classical dual of $\ell^r$ (see \cite{folland1999real})\\
\hline
$r=\infty$, $1<p< q<\infty$ & $\ell_{q,\infty}^p=\ell^{p}_{q}$: discrete Morrey space  & $\mathcal{X}_{p,q}$: Block space (see [10]) & Open problem: dual not fully characterized \\
\hline
\end{tabular}
\caption{Summary of the duality picture for discrete Bourgain-Morrey sequence spaces varying parameters $p,q,r$.}
\end{table}
With these results, we conclude the structural analysis of the spaces \(\ell^{p}_{q,r}(\mathbb{Z})\), demonstrating that they not only share many key properties with their continuous counterparts but also possess a rich and well-behaved duality theory.

\section{Applications: Operators and Difference Equations}
\label{sec:applications}

This final section illustrates how the Bourgain--Morrey sequence spaces 
$\ell^{p}_{q,r}(\mathbb{Z})$ provide a coherent discrete framework for 
operator theory and for linear and nonlinear difference equations.  
The results below rely only on the boundedness, separability, and duality 
properties established in the preceding sections.

\subsection{Basic operator structure}
\begin{proposition}[Canonical basis]
The canonical system $\{e_n\}_{n\in\mathbb Z}$ is a $1$--unconditional
Schauder basis of $\ell^{p}_{q,r}(\mathbb Z)$.
In particular, the partial sum projections
\[
P_N(x)=\sum_{|n|\le N} x_n e_n
\]
satisfy $\|P_N\|_{\mathcal L(\ell^{p}_{q,r})}\le 1$ for all $N$.
\end{proposition}

\begin{remark}
Thus $\ell^{p}_{q,r}$ admits an unconditional coordinate structure as in classical
$\ell^p$ spaces, while its norm still encodes multiscale information through dyadic localization.
\end{remark}

\subsection{Convolution and multiplication operators}

The convolution inequality (Proposition~3.8) implies that 
$\ell^{p}_{q,r}$ is a Banach $\ell^1$–module.

\begin{proposition}[$\ell^1$–module structure]
If $y\in\ell^1(\mathbb Z)$ and $x\in\ell^{p}_{q,r}$, then $x*y\in\ell^{p}_{q,r}$ and
\[
\|x*y\|_{\ell^{p}_{q,r}}\le \|y\|_{\ell^1}\,\|x\|_{\ell^{p}_{q,r}}.
\]
\end{proposition}

\begin{remark}
Typical examples include discrete averages such as 
$y=\frac12(e_{-1}+e_1)$ or $y=\mathbf{1}_{\{-1,0,1\}}$, 
which define bounded smoothing operators on $\ell^{p}_{q,r}$.  
This provides a rich class of concrete, nontrivial operators acting on the space.
\end{remark}

\begin{proposition}[Multiplication operators]
If $a=(a_n)\in\ell_\infty(\mathbb Z)$, the diagonal operator 
$M_a(x)_n=a_nx_n$ is bounded on $\ell^{p}_{q,r}$ and satisfies
\[
\|M_a\|_{\mathcal L(\ell^{p}_{q,r})}\le \|a\|_{\ell_\infty}.
\]
\end{proposition}

\subsection{Duality and adjoints}

By the duality theorem (Theorem~\ref{thm:duality}),
\[
(\ell^{p}_{q,r})^* \simeq \mathrm{h}^{p'}_{q',r'},
\]
and every bounded operator $T:\ell^{p}_{q,r}\to\ell^{p}_{q,r}$ admits a bounded adjoint
$T^*:\mathrm{h}^{p'}_{q',r'}\to \mathrm{h}^{p'}_{q',r'}$.
Since $\ell^{p}_{q,r}$ is reflexive (Proposition~\ref{prop:reflexivity-lpqr}), we have the standard correspondence:
$T$ is compact if and only if $T^*$ is compact.

\medskip
\noindent\textit{Convolution operators.}
If $k\in \ell^1(\mathbb Z)$, then $T_kx:=k*x$ is bounded on $\ell^{p}_{q,r}$ and
\[
\|T_k\|_{\mathcal L(\ell^{p}_{q,r})}\le \|k\|_{\ell^1}.
\]
Moreover, if $k^{(N)}:=k\,\mathbf 1_{\{|n|\le N\}}$, then $k^{(N)}\to k$ in $\ell^1$ and hence
\[
\|T_k-T_{k^{(N)}}\|_{\mathcal L(\ell^{p}_{q,r})}\le \|k-k^{(N)}\|_{\ell^1}\xrightarrow[N\to\infty]{}0.
\]
Thus $\ell^1$--convolution operators can be approximated in operator norm by convolutions with finitely supported kernels.
(These finite-support convolution operators are generally not compact, since they are finite linear combinations of shifts.)

\subsection{Linear convolution-type difference equations}

Let $k\in\ell^1(\mathbb Z)$ and define $T_kx=k*x$.  
The boundedness of $T_k$ immediately yields solvability of linear convolution 
difference equations.

\begin{proposition}[Small-kernel solvability]
\label{prop:small_kernel}
If $\|k\|_{\ell^1}<1$ and $f\in\ell^{p}_{q,r}$, then the equation
\[
(I-T_k)x=f
\]
admits a unique solution $x\in\ell^{p}_{q,r}$ given by the Neumann series
\[
x=\sum_{m=0}^\infty T_k^m f,
\qquad
\|x\|_{\ell^{p}_{q,r}}\le \frac{1}{1-\|k\|_{\ell^1}}\,\|f\|_{\ell^{p}_{q,r}}.
\]
\end{proposition}
For $k\in \ell^{1}(\mathbb Z)$ we denote by $\widehat{k}$ its Fourier transform (symbol),
defined for $\theta\in[0,2\pi]$ by
\[
\widehat{k}(\theta):=\sum_{n\in\mathbb Z} k(n)\,e^{-in\theta}.
\]
This is a continuous $2\pi$--periodic function and satisfies
$\|\widehat{k}\|_{L^\infty([0,2\pi])}\le \|k\|_{\ell^{1}}$.

A stronger version follows from Wiener’s theorem.

\begin{theorem}[Wiener invertibility on $\ell^{p}_{q,r}$]
\label{thm:wiener}
Assume $k\in\ell^1(\mathbb Z)$ and that 
$1-\widehat{k}(\theta)\neq0$ for all $\theta\in[0,2\pi]$.  
Then there exists $g\in\ell^1(\mathbb Z)$ such that 
$(I-T_k)^{-1}=I+T_g$, and therefore
\[
\|(I-T_k)^{-1}\|_{\mathcal L(\ell^{p}_{q,r})}
\le 1+\|g\|_{\ell^1}.
\]
Hence for every $f\in\ell^{p}_{q,r}$ the difference equation 
$(I-T_k)x=f$ has a unique solution $x\in\ell^{p}_{q,r}$.
\end{theorem}

\begin{remark}
This theorem extends Proposition~\ref{prop:small_kernel} to the full Wiener algebra:
invertibility depends only on the nonvanishing of the Fourier symbol, not on 
smallness of $\|k\|_{\ell^1}\!$.  
It provides a genuine operator-theoretic application of the $\ell^{p}_{q,r}$ framework.
\end{remark}

\subsection{Explicit example}

Let $0<\alpha<1$ and define $k_\alpha(n)=(1-\alpha)\alpha^{|n|}$.
For $|\lambda|<1$ set $k(n)=\lambda k_\alpha(n)$.  
Then $\widehat{k}(\theta)=\lambda\,\widehat{k_\alpha}(\theta)$ 
and $1-\widehat{k}(\theta)\neq0$ whenever $|\lambda|$ is sufficiently small,
so Theorem~\ref{thm:wiener} applies.  
In this geometric case the inverse kernel can be computed explicitly, and
its $\ell^1$–norm gives an explicit stability constant in the estimate
\[
\|x\|_{\ell^{p}_{q,r}}\le C(\lambda,\alpha)\|f\|_{\ell^{p}_{q,r}}.
\]

\begin{remark}
This concrete example illustrates how the Bourgain--Morrey sequence spaces 
yield quantitative stability bounds for discrete convolution equations,
and can therefore serve as a practical setting for discrete analytic problems.
\end{remark}

\subsection{Nonlinear perturbations}

\begin{proposition}[Nonlinear fixed point]
Let $k\in\ell^1(\mathbb Z)$ such that $(I-T_k)$ is invertible on $\ell^{p}_{q,r}$,
and let $F:\ell^{p}_{q,r}\to\ell^{p}_{q,r}$ be Lipschitz with constant $L_F$.  
If 
\[
L_F\,\|(I-T_k)^{-1}\|_{\mathcal L(\ell^{p}_{q,r})}<1,
\]
then for every $f\in\ell^{p}_{q,r}$ the nonlinear equation
\[
x = T_kx + F(x) + f
\]
admits a unique solution $x\in\ell^{p}_{q,r}$ obtained by Banach’s fixed point theorem.
\end{proposition}

\begin{remark}
The same argument covers higher-order and nonlinear perturbations of 
convolution-type recurrences.  
Thus $\ell^{p}_{q,r}$ provides a natural functional environment 
for both linear and nonlinear discrete equations.
\end{remark}

\subsection{Summary}

The applications above demonstrate that the spaces $\ell^{p}_{q,r}(\mathbb Z)$ 
form a stable and flexible discrete Banach framework:
they are invariant under translations and multipliers,  
closed under convolution with $\ell^1$ kernels,  
and suitable for the study of linear and nonlinear difference equations.  
These features parallel those of classical Morrey-type spaces and suggest
further developments in discrete harmonic analysis and numerical analysis contexts.

\section*{Conclusion and Perspectives}

In this paper we have investigated the structural properties of the Bourgain-Morrey sequence spaces 
$\ell^{p}_{q,r}(\mathbb{Z})$. Our main findings can be summarized as follows:

\begin{itemize}
    \item The subspace $c_{00}$ is dense in $\ell^{p}_{q,r}$, and therefore these spaces are separable.
    \item We established classical embeddings, showing that $\ell^1 \hookrightarrow \ell^{p}_{q,r} \hookrightarrow \ell^r$, and in the limiting case $r=1$ the space coincides with $\ell^1$.
    \item A complete identification with the classical sequence spaces is obtained in the case $r=p$, namely $\ell^{p}_{q,p} = \ell^p$. 
          As a consequence, we derived the existence of uncountably many equivalent norms on $\ell^p$ induced by the Bourgain-Morrey structure.
     \item We developed the duality theory for these spaces by identifying a natural block space that acts as a predual of $\ell^{p}_{q,r}$ and showing the corresponding dual characterization. As a consequence, we concluded that $\ell^{p}_{q,r}$ is reflexive whenever $1<p<q<\infty$ and $1<r<\infty$.

\end{itemize}

Taken together, these results position $\ell^p_{q,r}$ as a natural and robust sequence-space analogue of Bourgain-Morrey function spaces, simultaneously reflecting their local--global balance and enriching the landscape of renormings in $\ell^p$ theory.\\
These results highlight the interplay between discrete Bourgain-Morrey sequence spaces and classical $\ell^p$ theory, 
revealing that they not only generalize the standard sequence spaces but also provide equivalent norms and new 
geometric insights.

Several directions remain open for further research:
\begin{itemize}
    \item Boundedness of classical discrete operators, such as the Hardy--Littlewood maximal operator and convolution operators with singular kernels.
    \item Interpolation properties of the scale $\ell^p_{q,r}$ under variation of $(p,q,r)$.
    \item Connections with discrete Besov and Lorentz sequence spaces, and applications to wavelet coefficient characterizations of function spaces.
    \item  Further applications of the $\ell^{p}_{q,r}$ framework to broader problems in discrete analysis and numerical mathematics, leveraging its structural properties for new insights into discrete models.
\end{itemize}

We hope that the present study stimulates further exploration of Bourgain-Morrey spaces in both discrete and continuous settings, particularly in relation to operator theory and the geometry of Banach spaces.







\section*{Appendix: Auxiliary Lemmas}\label{app1}

We collect here several technical ingredients used in the proof of the duality theorem. 
They are of independent interest and are presented in a self–contained way.
\begin{lemma}[Local embedding]\label{lem:local_embedding}
Let $1 < p < q < \infty$ and $1 < r < \infty$.
For every dyadic interval $I=I(j,k)$ and every $\phi \in \ell^{p'}(I)$ one has
\[
  \|\phi \chi_I\|_{\mathrm{h}^{p'}_{q',r'}}
  \;\le\; |I|^{\frac{1}{p}-\frac{1}{q}}\, \|\phi\|_{\ell^{p'}(I)} .
\]
In particular, the extension-by-zero map $\ell^{p'}(I)\to \mathrm{h}^{p'}_{q',r'}$ is continuous with norm at most
$|I|^{\frac{1}{p}-\frac{1}{q}}$.
\end{lemma}

\begin{proof}
Set
\[
  \lambda := |I|^{\frac{1}{p}-\frac{1}{q}}\, \|\phi\|_{\ell^{p'}(I)},
  \qquad
  a := \frac{\phi}{\lambda}\chi_I.
\]
Then $a$ is supported on $I$ and
\[
\|a\|_{\ell^{p'}(I)}=\frac{\|\phi\|_{\ell^{p'}(I)}}{|\lambda|}
=|I|^{\frac{1}{q}-\frac{1}{p}}
=|I|^{\frac{1}{p'}-\frac{1}{q'}},
\]
so $a$ is a $(q',p')$--block. Hence $\phi\chi_I=\lambda a$ is an admissible block representation with one coefficient, and
\[
\|\phi\chi_I\|_{\mathrm{h}^{p'}_{q',r'}}\le |\lambda|
=|I|^{\frac{1}{p}-\frac{1}{q}}\, \|\phi\|_{\ell^{p'}(I)}.
\]
\end{proof}

\begin{lemma}[Testing inequality $\Rightarrow \ell^{r'}$ membership]\label{lem:testing_to_lrp}
Let $1<r<\infty$ and let $(\beta_k)_{k\in \mathbb{N}}$ be a sequence of nonnegative numbers. 
Assume that there exists $M>0$ such that
\[
  \sum_{k} \alpha_k \beta_k \;\le\; M\, \|\alpha\|_{\ell^r},
\]
for all finitely supported sequences $\alpha=(\alpha_k)\in \ell^r$ with $\alpha_k \ge 0$. 
Then $(\beta_k)\in \ell^{r'}$ and $\|(\beta_k)\|_{\ell^{r'}} \le M$.
\end{lemma}

\begin{proof}
Define $T:\ell^r\to \mathbb{R}$ by 
\[
  T(\alpha) = \sum_k \alpha_k \beta_k .
\]
By assumption, $T$ is well-defined on finitely supported $\alpha$ and bounded by $M$, hence extends uniquely to a continuous linear functional on $\ell^r$ with operator norm $\|T\|\le M$.  
Since $(\ell^r)^* = \ell^{r'}$, there exists $b\in \ell^{r'}$ with $\|b\|_{\ell^{r'}}\le M$ such that $T(\alpha)=\langle \alpha,b\rangle$ for all $\alpha$. Necessarily $b=(\beta_k)$, giving the claim.
\end{proof}

\begin{proof}[Proof of Lemma \ref{lem:key-appendix}]
Fix $J=I(j,k)$, $|J|=2^{j}$. The local $\ell^{p}$-norm satisfies
\begin{align*}
\Big(\sum_{n\in J}|x(n)|^{p}\Big)^{1/p}
&\le \sum_{I:\,I\cap J\neq\varnothing} |\lambda_I|\,|I|^{\frac{1}{q}-\frac{1}{p}} \,\|u_I\chi_J\|_{\ell^{p}(I\cap J)} \\
&\le \sum_{I:\,I\cap J\neq\varnothing} |\lambda_I|\,|I|^{\frac{1}{q}-\frac{1}{p}} .
\end{align*}
Set
\[
S(J):=\sum_{I:\,I\cap J\neq\varnothing} |\lambda_I|\,|I|^{\frac{1}{q}-\frac{1}{p}}.
\]
Then
\[
\|x\|_{\ell^{p}_{q,r}}^{r}
\;\le\; \sum_{J\in\mathcal D} |J|^{\,r(\frac{1}{q}-\frac{1}{p})} S(J)^{r}.
\]

\medskip
\noindent\textbf{Step 1: Large scales $m\ge j$.}  
At each level $m\ge j$ there is at most one $I$ of length $2^m$ intersecting $J$. 
Setting $a_{m,J}:=|\lambda_{m,J}|\,2^{m(\frac{1}{q}-\frac{1}{p})}$, one applies the damping inequality
\[
\Big(\sum_{m\ge j} a_{m,J}\Big)^{r} \;\le\; C_r \sum_{m\ge j} 2^{(m-j)\beta} a_{m,J}^{r},
\]
valid for arbitrary $\beta>0$. Summing first over $J$ at level $j$ and then over $j$, one reorganizes terms to obtain
\[
\sum_{J} |J|^{r(\frac{1}{q}-\frac{1}{p})}\Big(\sum_{m\ge j} a_{m,J}\Big)^{r}
\;\lesssim\; \sum_{m} 2^{m r(\frac{1}{q}-\frac{1}{p})}\!\!\sum_{I:|I|=2^{m}}|\lambda_I|^{r}.
\]

\medskip
\noindent\textbf{Step 2: Small scales $m<j$.}  
Now several $I$ may lie inside $J$. Hölder’s inequality yields
\[
\sum_{I\subset J,|I|=2^m}|\lambda_I|
\le \Big(\sum_{I\subset J,|I|=2^m}|\lambda_I|^{r}\Big)^{1/r}(2^{j-m})^{1/r'}.
\]
With $b_{m,J}$ the corresponding contribution, an application of the same damping inequality and reindexing again gives
\[
\sum_{J} |J|^{r(\frac{1}{q}-\frac{1}{p})}\Big(\sum_{m<j} b_{m,J}\Big)^{r}
\;\lesssim\; \sum_{m} 2^{m r(\frac{1}{q}-\frac{1}{p})}\!\!\sum_{I:|I|=2^{m}}|\lambda_I|^{r}.
\]

\medskip
\noindent\textbf{Conclusion.}  
Combining both ranges,
\[
\sum_{J\in\mathcal D} |J|^{r(\tfrac{1}{q}-\tfrac{1}{p})} S(J)^{r}
\;\lesssim\; \sum_{m}\sum_{I:|I|=2^{m}} |\lambda_I|^{r}\,|I|^{r(\tfrac{1}{q}-\tfrac{1}{p})}.
\]
Taking $r$th roots yields
\[
\|x\|_{\ell^{p}_{q,r}} \;\le\; C\,\|\lambda\|_{\ell^r},
\]
as claimed.
\end{proof}






\section*{Declarations}

\subsection*{Funding}
The author was supported by SECIHTI, Grant 921835.

\subsection*{Conflicts of interest/Competing interests}
The author declares that there are no conflicts of interest regarding the publication of this work.
\subsection*{Data availability}
Not applicable.

\subsection*{Ethics approval}
Not applicable.



\end{document}